\documentclass{article}
\usepackage{amssymb,amsmath,amsthm}
\newtheorem{theorem}{Theorem}[section]
\newtheorem{lemma}[theorem]{Lemma}
\begin{document}

\author{Georgii Khantarzhiev}
\title{\textbf{Application of difference sequences theory}} 
\date{January 14, 2014}
\maketitle


\phantom{x}


\begin{abstract}

The results of difference sequences theory are applied to analytic function theory and Diophantine equations.
As a result we have the equation which connects the $n$-th derivative of a function with the difference sequence for the values of this function. Also the
results of difference sequences theory helps to discover some features of the whole kind of Diophantine equations.
The method presented allows to find limits where Diophantine equation does not have integer solutions. The higher
power of Diophantine equation the better this method works.

\end{abstract}

\begin{center}
\section{Introduction}
\end{center}

I give the short introduction to difference sequences theory below to show the main idea.
Let's have a look at the following system:

$$
\setcounter{MaxMatrixCols}{20}
\begin{matrix}
    &&&&& \hdotsfor[3]{3} \\
    &&&& \vartriangle^{3} \! u_{0}, && \vartriangle^{3} \! u_{1}, && \ldots\\
    &&& \vartriangle^{2} \! u_{0}, && \vartriangle^{2} \! u_{1}, && \vartriangle^{2} \! u_{2}, &&\ldots\\
    && \vartriangle \! u_{0}, && \vartriangle \! u_{1}, && \vartriangle \! u_{2}, && \vartriangle \! u_{3}, && \ldots\\
    & u_{0}, && u_{1}, && u_{2}, && u_{3}, && u_{4}, && \ldots\\
\end{matrix}
$$
\\
where
\begin{align*}
\phantom{x}&  & \vartriangle \!\! u_{k}& = u_{k+1} - u_{k}, & u_{k}& = \, \vartriangle^{0} \!\! u_{k}, & \phantom{x}&  \\
\phantom{x}&  & \vartriangle^{n+1} \!\! u_{k}& = \, \vartriangle^{n} \!\! u_{k+1} \, - \vartriangle^{n} \!\! u_{k} & k& = 0, 1, 2, 3, \ldots\ & \phantom{x}&
\end{align*}

In accordance with the above definitions let's take a look at the following number sequences:

$$
u_{k} = k\,,\phantom{^{2}} \; \; \; \; k = 0, 1, 2, 3, \ldots
$$

$$
\setcounter{MaxMatrixCols}{20}
\begin{matrix}
&& 1, && 1, && 1, && 1, && 1, && 1, && \ldots\\
& 0, && 1, && 2, && 3, && 4, && 5, && 6, && \ldots\\
\end{matrix}
$$

$$
u_{k} = k^{2}\,, \; \; \; \; k = 0, 1, 2, 3, \ldots
$$

$$
\setcounter{MaxMatrixCols}{20}
\begin{matrix}
&&& 2, && 2, && 2, && 2, && 2, && \ldots\\
&& 1, && 3, && 5, && 7, && 9, && 11, && \ldots\\
& 0, && 1, && 4, && 9, && 16, && 25, && 36, && \ldots\\
\end{matrix}
$$

$$
u_{k} = k^{3}\,, \; \; \; \; k = 0,1,2,3, \ldots
$$
\\
$$
\setcounter{MaxMatrixCols}{20}
\begin{matrix}
&&&& 6, && 6, && 6, && 6, && \ldots\\
&&& 6, && 12, && 18, && 24, && 30, && \ldots\\
&& 1, && 7, && 19, && 37, && 61, && 91, && \ldots\\
& 0, && 1, && 8, && 27, && 64, && 125, && 216, && \ldots\\
\end{matrix}
$$
etc. any $u_{k}=k^{n}$, $n\in\mathbb{N}$.
We can suppose that for the following number sequences: $u_{k} = k^{n}$, $k\in\mathbb{Z},$
$$
\vartriangle^{n} \!\! u_{k} = n!\,, \; \; \; \;  n = 1, 2, 3, \ldots 
$$

Newton researched such difference sequences and proved the following theorem

                             \begin{theorem}

For polynomial $P_{n}(x) = a_{0} x^{n} \, + \, a_{1} x^{n-1} \, + \cdots \, + \, a_{n},$ where $x\in\mathbb{R}, \: a_{0}, a_{1}, \ldots , a_{n} \in\mathbb{R}, \: a_{0} \neq 0, \: n\in\mathbb{N},$
$k$ is the step of difference sequence, $k\in\mathbb{R} , \: k \neq 0$

\begin{equation} \label{1}
{\vartriangle^{n} \!\! P_{n}(x)} = \sum^{n}_{i = 0} \binom{n}{i} (-1)^{i} P_{n} (x - ki) = a_{0} k^{n} n!
\end{equation}

                             \end{theorem}

For example you can familiarize with difference sequences theory in ~\cite{Combinatorics}.

\begin{center}
\item{
\section{Analytic Functions}
}
\end{center}

Apply this result to analytic functions. Let $f(x)$ be a real analytic function in the interval $(a,b)$. In such case by Taylor's theorem it can be written 

$$
f(x) = f(x_{0}) + \frac{f'(x_{0})}{1!}(x - x_{0}) + \cdots + \frac{f^{(n)}(x_{0})}{n!}(x - x_{0})^n + \frac{f^{(n+1)}(\xi)}{(n+1)!}(x - x_{0})^{n+1},
$$
where $x,x_{0},\xi\in(a,b)$ and the remainder has Lagrangian form.    

For $i=0,1,2,\ldots, n$ and $k\in\mathbb{R} , \: k \neq 0$ : $x + ik\in(a,b)$

\begin{multline*}
f(x+ik) = f(x_{0}) + \frac{f'(x_{0})}{1!}(x - x_{0} + ik) + \cdots \\
+ \frac{f^{(n)}(x_{0})}{n!}(x - x_{0} + ik)^{n} + \frac{f^{(n+1)}(\xi_{i})}{(n+1)!}(x - x_{0} + ik)^{n+1},
\end{multline*}
where $\xi_{i}\in(x_{0}, x+ik).$ Then we can put

$$
P'_{n}(x+ik) = \frac{f^{(n)}(x_{0})}{n!}(x - x_{0} + ik)^{n} + \cdots + \frac{f'(x_{0})}{1!}(x - x_{0} + ik) + f(x_{0}),
$$
hence

\begin{multline*}
\sum^{n}_{i=0}\binom{n}{i} (-1)^{n-i} f(x+ik) = \sum^{n}_{i=0}\binom{n}{i} (-1)^{n-i} P'_{n}(x+ik) + \\
+ \sum^{n}_{i=0}\binom{n}{i} (-1)^{n-i} \frac{f^{(n+1)}(\xi_{i})}{(n+1)!}(x - x_{0} + ik)^{n+1},
\end{multline*}

by Theorem 1.1 and taking into account the sign at \eqref{1}

$$
\sum^{n}_{i=0}\binom{n}{i}(-1)^{n-i} P'_{n}(x+ik) = f^{(n)}(x_{0})k^{n},
$$

$$
f^{(n)}(x_{0})(-k)^{n} = \sum^{n}_{i=0}\binom{n}{i} (-1)^{i} f(x+ik) - \sum^{n}_{i=0}\binom{n}{i} (-1)^{i} \frac{f^{(n+1)}(\xi_{i})}{(n+1)!}(x - x_{0} + ik)^{n+1}.
$$

This equation is the generalized form of \eqref{1} for functions.
The equation shows a connection between the $n$-th derivative of a function in a point $x_{0}$ and the difference sequence for this function which begins in other point $x$.

%

In the case $x = x_{0}$ we have

$$
f^{(n)}(x_{0})(-k)^{n} = \sum^{n}_{i=0}\binom{n}{i} (-1)^{i} f(x_{0}+ik) - \sum^{n}_{i=1}\binom{n}{i} (-1)^{i} \frac{f^{(n+1)}(\xi_{i})}{(n+1)!}(ik)^{n+1}.
$$


\begin{center}
\item{
\section{Diophantine equations}
}
\end{center}

Lets apply the result of difference sequences theory to Diophantine equations. For example lets take Fermat's Last Theorem ~\cite{Fermat's Last Theorem} which was proven by Wiles in 1995 ~\cite{Wiles}.
Fermat's Last Theorem states that no there positive integers $x,y$ and $z$ can satisfy the equation $x^{n} + y^{n} = z^{n}$ for any integer value of $n$
greater than two.

Let $x > y$. For $x^{n}$ we can write

$$
n! = \sum^{n}_{i=0} \binom{n}{i} (-1)^{n-i} (x + i)^{n} = \sum^{n}_{i=0} \binom{n}{i} (-1)^{n-i} (x' + y + i)^{n},
$$
where $x' = x - y$. For $y^{n}$ we can write

$$
n! = \sum^{n}_{k=0} \binom{n}{k} (-1)^{n-k} (y + k)^{n},
$$

By adding these two sums we get the equation for $z^{n}$.

\begin{equation}\label{10}
2n! = \binom{n}{n} (-1)^{0} z^{n}_{y+n}(x') - \binom{n}{n-1} (-1)^{1} z^{n}_{y+n-1}(x') + \cdots + \binom{n}{0} (-1)^{n} z^{n}_{y}(x')
\end{equation}
where $z^{n}_{p}(x',n) = (x' + p)^{n} + p^{n}$, $p = 0,1,2,3 \ldots \; .$

Signify

$$
z_{p}(x',n) = \sqrt[n]{z^{n}_{p}(x',n)}.
$$

Notice that \eqref{10} is right for $\forall$  $x', p$. It means the function $z_{p}(x',n)$ behaves by certain way for $\forall$  $x', p$.
In \eqref{10} the coefficient before $n!$ is equal $2$. Lets have a look at \eqref{1} we have to take step $k = \sqrt[n]{2}$ in \eqref{1} to get coefficient
$2$ before $n!$ in the right part. The step of the function $z_{p}(x',n)$ is not certainly a constant, but it has to behave by certain way to satisfy \eqref{10}. The behaviour of this step may give us some information about possible integer roots of this Diophantine equation.



Lets have a look at the following function

$$
Step(x',p,n) = z_{p+1}(x',n) - z_{p}(x',n),
$$

where $x' \geq 1$, $(x>y)$. 

\begin{lemma}

Prove that

$$
1 < Step(x',p,n) < \sqrt[n]{2}.
$$

\end{lemma}
\begin{proof}

We fix the parameter $x'$ and the parameter $p$ is growing. First of all let's show that the function $Step(x',p,n)$ grows with the growing of $p$.




\begin{equation}\label{11}
\begin{split}
Step(x',p,n) &= z_{p+1}(x',n) - z_{p}(x',n) = \\
&= \sqrt[n]{(x'+p+1)^{n} + (p+1)^{n}} - \sqrt[n]{(x'+p)^{n} + p^{n}} = \\
&= (p+1)\sqrt[n]{\left(1+\frac{x'}{p+1}\right)^n + 1} \; \; - \;\; p\sqrt[n]{\left(1+\frac{x'}{p}\right)^n + 1},
\end{split}
\end{equation}

where $p \ne 0$.

Let's signify a factor after $p+1$ as $I_{1}$ and a factor after $p$ as $I_{2}$, $I_{1} < I_{2}$. The more $p$ grows the more $x'/(p+1)$
becomes closer to $x'/p$. Hence the more $p$ grows the more $I_{1}$ becomes closer to $I_{2}$. Hence the difference between
the first component and the second one in \eqref{11} grows. It means that the function $Step(x',p,n)$ grows with the growing of $p$.

Now let's show that for $p=1,2,3, \ldots$ function $Step(x',p,n) < \sqrt[n]{2}$ .


Since we have shown that the function $Step(x',p,n)$ grows with the growing of $p$, we put $p \to \infty$ to find the high limit of this function.

$$
\lim_{p \to \infty} Step(x',p,n) = \lim_{p \to \infty} \left [\sqrt[n]{(x' + p + 1)^n + (p + 1)^n} - \sqrt[n]{(x' + p)^n + p^n} \;  \right ] =
$$

signify $x' = \alpha p$, $0 < \alpha \leq 1$

$$
= \lim_{p \to \infty} \left[ (p(1 + \alpha) + 1) \sqrt[n]{1 + \left[ \frac{p + 1}{p(1 + \alpha) + 1} \right]^n} - p(1 + \alpha) \sqrt[n]{1 + {\left[ \frac{1}{1 + \alpha} \right]}^n} \;  \right] =
$$

$$
= \lim_{p \to \infty} \left[ (p(1 + \alpha) + 1) \sqrt[n]{1 + \left[ \frac{1}{1 + \alpha} \right]^n} - p(1 + \alpha) \sqrt[n]{1 + {\left[ \frac{1}{1 + \alpha} \right]}^n} \;  \right] =
$$

\begin{equation}\label{12}
= \sqrt[n]{1 + \left[ \frac{1}{1 + \alpha} \right]^n}.
\end{equation}

Signify the expression in \eqref{12} like $f(\alpha)$. Since $f^{'}_{\alpha}(\alpha) < 0$, the expression in \eqref{12} reduces when $\alpha$ grows.
Hence $Step(x',p,n) < \sqrt[n]{2}$, $(\alpha \to 0)$.

Then let's show that for $p=1,2,3, \ldots$ function $Step(x',p,n) > 1$.


So function $Step(x',p,n)$ grows with the growing of $p$, it means that the low limit of $Step(x',p,n)$ reaches when $p=1$. In other words,
if we prove that $Step(x',1,n) > 1$ then $Step(x',p,n) > 1$ for $p > 0$.

$$
Step(x',1,n) = z_{2}(x',n) - z_{1}(x',n) = \sqrt[n]{(x' + 2)^n + 2^n} - \sqrt[n]{(x' + 1)^n + 1},
$$

Let's research functions $z_{2}(x',n)$ and $z_{1}(x',n)$ in the interval $x' \in [1, \infty)$ and the power $n$ is fixed. It is obviously that
these functions do not intersect each other. Also it is not hard to prove that the first and the second derivatives of these functions
are positive. It means that these functions grow and the velocity of their growing grows too. Taking these facts into account we can assert
that the function $Step(x',1,n)$ reaches its maximum and minimum values in the ends of the interval $x' \in [1, \infty)$.


Calculate the value of $Step(x',1,n)$ for $x' \to \infty$ to show that it is more than one.

$$
\lim_{x' \to \infty} \left[ \sqrt[n]{(x'+2)^n + 2^n} - \sqrt[n]{(x'+1)^n + 1} \; \right] =
$$

$$
\lim_{x' \to \infty} \left[ \; 1 + \frac{2^n}{n(x' + 2)^{n-1}} - \frac{1}{n(x' + 1)^{n-1}} + O \left( \frac{1}{ (x')^{2n-1} } \right) \; \right] > 1.
$$

Hence $Step(x',1,n) > 1$ and it means that $Step(x',p,n) > 1$ for $p=1,2,3, \ldots$. The same is right for $Step(x',0,n)$.

$$
Step(x',0,n) = \sqrt[n]{(x' + 1)^n + 1} - x' > 1.
$$

And the finale step - show that $Step(x',0,n) < \sqrt[n]{2}$,

$$
\sqrt[n]{(x' + 1)^n + 1} - x' < \sqrt[n]{2}.
$$

Obviously that

$$
\sum^{n-1}_{i=1} \binom{n}{i} (x')^{n-i} < \sum^{n-1}_{i=1} \binom{n}{i} (x')^{n-i} (\sqrt[n]{2})^i,
$$

Add $(x')^n + 2$ to the both parts of the inequality

$$
(x')^n + \sum^{n-1}_{i=1} \binom{n}{i} (x')^{n-i} + 1 + 1 < (x')^n + \sum^{n-1}_{i=1} \binom{n}{i} (x')^{n-i} (\sqrt[n]{2})^i + 2,
$$

$$
(x' + 1)^n + 1 < (x' + \sqrt[n]{2})^n,
$$

hence

$$
\sqrt[n]{(x' + 1)^n + 1} < x' + \sqrt[n]{2}.
$$

So we have the following

$$
1 < Step(x',p,n) < \sqrt[n]{2}, \;\; p = 0,1,2,3, \ldots \;.
$$

\end{proof}

It means that

\begin{equation}\label{13}
0 < \{ Step(x',p,n) \} < \sqrt[n]{2} - 1, \;\; p = 0,1,2,3, \ldots \;.
\end{equation}

In this case there is such $j\in\mathbb{N}$ :

$$
\begin{cases}
  \sum\limits^{j-2}_{k=0} \lbrace Step(x',k,n) \rbrace  < 1, \\
                                     \\
  \sum\limits^{j-2}_{k=0} \lbrace Step(x',k,n) \rbrace  +  \lbrace Step(x',j-1,n) \rbrace  \geq 1.
\end{cases}
$$

Hereby, if we start from integer $z_{0}(x',n) = x'$ then until $p$ reaches $j$ the number $z_{p}(x',n)$ can not be integer such

$$
\sum\limits^{j-2}_{k=0} \lbrace Step(x',k,n) \rbrace = \{ z_{j-1}(x',n) - z_{0}(x',n) \} < 1.
$$

Lets find the low limit of $j$ using \eqref{13}, when $p$ reaches $l\in\mathbb{N}$ :

$$
\begin{cases}
(l - 1)(\sqrt[n]{2} - 1) < 1, \\
l(\sqrt[n]{2} - 1) > 1,
\end{cases}
$$

$z_{l}(x',n)$  can already be an integer number. Hence the low limit of $j$ is $l$,


$$
l = \left\lceil \frac{1}{\sqrt[n]{2} - 1} \right\rceil, \; \; \; \; \; j \geq \left\lceil \frac{1}{\sqrt[n]{2} - 1} \right\rceil.
$$

Therefore we can put the following statements for Fermat's Last Theorem:

\newtheorem{statement}{Statement}

                             \begin{statement}


Fermat's Last Theorem is right for any $x, (x > y)$ and for every $y$ which satisfies the following inequality

$$
y <  \frac{1}{\sqrt[n]{2} - 1}, \, \, (x > y).
$$
                             \end{statement}

                             \begin{proof}

We start from integer $z_{0}(x',n) = x'$ and $p$ grows. It has been proven before that if $p$ does not reach $j$
then $z_{p}(x',n)$ can not be an integer number. Hence $z_{p}(x',n) = \sqrt[n]{(x' + p)^n + p^n}$ is not an integer number when

$$
p <  \frac{1}{\sqrt[n]{2} - 1}.
$$

If this is proven for $p$ hence it is proven for $y$ :

$$
y <  \frac{1}{\sqrt[n]{2} - 1}.
$$

So for any $x$ and for such $y$ the equation $x^n + y^n = z^n$, $(x > y)$ does not have integer roots.

                             \end{proof}

The higher power $n$ the more $y$ can be. It means that the method works more sensible with
the growing of power $n$.

                             \begin{statement}


If the equation

\begin{equation}\label{14}
z^{n}_{p}(x',n) = (x' + p)^{n} + p^{n},
\end{equation}

where $x', p \geq 1$, has two integer roots $z_{k}(x',n)$ and $z_{k+i}(x',n)$ then

$$
i \geq \left\lceil \frac{1}{\sqrt[n]{2} - 1} \right\rceil.
$$

                             \end{statement}

                             \begin{proof}

We have $z_{k}(x',n)$ - the integer root of \eqref{14}, hence

$$
\left\{ \sum^{k-1}_{p=0} Step(x',p,n) \right\} = 0,
$$


By definition $i$ :

\begin{equation}\label{15}
\left\{ \sum^{k+i-1}_{p=k} Step(x',p,n) \right\} = 0,
\end{equation}

The sum in \eqref{15} consists of $i$ components $Step(x',p,n)$.
Taking into account \eqref{13} there is such $l\in\mathbb{N}$ :

$$
\begin{cases}
(l - 1)(\sqrt[n]{2} - 1) < 1, \\
l(\sqrt[n]{2} - 1) > 1.
\end{cases}
$$

It is obviously from \eqref{13} that the sum in \eqref{15} can not consist
of number of components less than $l$. Hence

$$
i \geq \left\lceil \frac{1}{\sqrt[n]{2} - 1} \right\rceil.
$$


                             \end{proof}

                             \begin{statement}


When $\{ z_{k}(x',n) \}$ is less than $( \! \sqrt[n]{2} - 1)j < 1$, $j\in\mathbb{N}$,
then $z_{k+i}(x',n)$ can be the integer root of \eqref{14} if

$$
i \geq \left\lceil \frac{1}{\sqrt[n]{2} - 1} \right\rceil - j .
$$

                             \end{statement}

                             \begin{proof}

We have $\{ z_{k} (x',n) \} < (\sqrt[n]{2} - 1)j$, it means that



$$
\left\{ \sum^{k-1}_{p=0} Step(x',p,n) \right\} < (\sqrt[n]{2} - 1)j.
$$

If $z_{k+i} (x',n)$ is the integer root of \eqref{14}, then there is $i$ :

\begin{equation}\label{16}
\left\{ \sum^{k-1}_{p=0}  Step(x',p,n) \right\} + \left\{ \sum^{k+i-1}_{p=k}  Step(x',p,n) \right\} = 1.
\end{equation}

Taking into account \eqref{13} there is such $l\in\mathbb{N}$ :

$$
\begin{cases}
(l - 1 + j)(\sqrt[n]{2} - 1) < 1, \\
(l + j)(\sqrt[n]{2} - 1) > 1.
\end{cases}
$$

$$
l + j = \left\lceil \frac{1}{\sqrt[n]{2} - 1} \right\rceil.
$$

It is obviously from \eqref{13} that the second sum in \eqref{16} can not consist
of number of components less than $l$. Hence $i \geq l$ and

$$
i \geq \left\lceil \frac{1}{\sqrt[n]{2} - 1} \right\rceil - j.
$$




                             \end{proof}

We can say that $z^{n}_{p}(x',n)$ from \eqref{14} is a branch of Diophantine equation $z^{n} = x^{n} + y^{n}$ and
the fragmentation into branches is carried out by $p$ parameter.

This method is applicable to the whole kind of Diophantine equations. The more power $n$ the more effective this
method works.

There is a problem in the application of this method to various Diophantine equations, it is a
necessity to find the limits of function $Step (x',  ... \; p , n)$ for each case. For example let's take the
following Diophantine equation

$$
z^{n} = A x^{n} + y^{n}, \; \; x > y, \; \; A \in\mathbb{N}.
$$

If we construct for this Diophantine equation an equation like \eqref{10} it will hint that the high limit of
$Step(x', p, n)$ is equal to $\sqrt[n]{A + 1}$. But it has to be proven for this Diophantine equation like for
the previous one. How not to prove it in any case? Is there any way to ease this procedure? Maybe there is a
way to avoid this procedure using \eqref{1} and \eqref{10} to find the limits of \emph{Step}-function.

\phantom{x}
\phantom{x}

\phantom{x}

\textit{E-mail address}: \texttt{khantarzhiev@gmail.com}

\phantom{x}

\textsc{Moscow, Russia}

\end{document}